\documentclass[a4paper,12pt]{amsart}
\usepackage{geometry}
\usepackage[latin1]{inputenc}
\usepackage[italian,english]{babel}
\usepackage{amsmath, amsfonts, amsthm, xcolor}
\usepackage{paralist}
\numberwithin{equation}{section}
\usepackage{mathtools, mathabx, verbatim}
\usepackage{MnSymbol}
\usepackage{multicol}

\newtheorem{thm}{Theorem}[section]
\newtheorem{lem}[thm]{Lemma}
\newtheorem{cor}[thm]{Corollary}
\newtheorem{prop}[thm]{Proposition}

\newtheorem{defn}[thm]{Definition}
\theoremstyle{definition}
\newtheorem{rem}[thm]{Remark}
\theoremstyle{remark}

\newcommand{\ds}{\displaystyle}

\newcommand{\abs}[1]{\left\vert#1\right\vert}

\newcommand{\R}{\mathbb{R}}

\newcommand{\de}{\partial}

\DeclareMathOperator{\spt}{spt}

{\left\{\begin{array}{@{}l@{}}}{\end{array}\right.}
\patchcmd{\abstract}{\scshape\abstractname}{\textbf{\abstractname}}{}{}
\makeatletter 
\def\@makefnmark{} 
\makeatother 

\title[First and Second Steklov-Dirichlet eigenvalues]{Estimates for the first and second Steklov-Dirichlet eigenvalues}
\author[R. Sannipoli]{Rossano Sannipoli}
\address{Dipartimento di Matematica ``Tullio Levi-Civita", Universit\'a degli Studi di Padova, Via Trieste 63, 35131 Padua, Italy.}
\email{rossano.sannipoli@unipd.it}

\begin{document}
\begin{abstract}
In this paper, we deal with the Steklov-Dirichlet eigenvalue problem for the Laplacian in annular domains. More precisely, we consider \( \Omega_r = \Omega_0 \setminus \overline{B}_r \), where \( \Omega_0 \subset \mathbb{R}^n \), \( n \geq 2 \), is an open, bounded set with a Lipschitz boundary, and \( B_r \) is the ball centered at the origin with radius \( r > 0 \), such that \( \overline{B}_r \subset \Omega_0 \). In the first part of the paper, we focus on the first Steklov-Dirichlet eigenvalue \( \sigma_1(\Omega_r) \) and prove that the sequence of corresponding normalized eigenfunctions converges to a particular constant as \( r \to 0^+ \). This will allow us to prove an isoperimetric inequality for \( \sigma_1(\Omega_r) \) when \( r \) is small enough, under a measure constraint. The second part is focused on the second Steklov-Dirichlet eigenvalue \( \sigma_2(\Omega_r) \). We prove that it converges to the first non-trivial Steklov eigenvalue \( \overline{\sigma}_1(\Omega_0) \) of the non-perforated domain \( \Omega_0 \). This result, together with the Brock and Weinstock inequalities, respectively, allows us to prove two isoperimetric inequalities for small holes.\\ 

\end{abstract}
\maketitle

\section{Introduction}
Let $\Omega_0$ be an open, bounded and connected set with Lipschitz boundary and let us consider a ball $B_r$ of radius $r>0$ centered at the origin, such that $\overline{B}_r \subset \Omega_0$. In $\Omega= \Omega_0\setminus B_R$ we consider the following Steklov-Dirichlet eigenvalue problem for the Laplacian 
\begin{equation}\label{proSD}
\begin{cases}
\Delta u=0 & \mbox{in}\ \Omega\vspace{0.2cm}\\
\dfrac{\de u}{\de \nu}=\sigma u&\mbox{on}\ \partial\Omega_0\vspace{0.2cm}\\ 
u=0&\mbox{on}\ \partial B_{r}, 
\end{cases}
\end{equation}
where $\nu$ is the outer unit normal to $\partial \Omega$. The spectrum of \eqref{proSD} is related to the one of the so-called \textit{Neumann-to-Dirichlet} operator, which is compact and self-adjoint. The spectrum is discrete, positively divergent and the sequence of eigenvalues can be ordered as follows 
(see for instance \cite{Agr2006})
\[
0<\sigma_{1}(\Omega)\leq\sigma_2(\Omega)\leq... \nearrow +\infty.
\]

\noindent In particular, for $k\ge 1$, the $k$-th eigenvalue can be computed applying the usual min-max Theorem (see \cite{Agr2006,michetti2022steklov,BCV2023})
\begin{equation}\label{eq:carvarSD}
    \sigma_k(\Omega)=
\inf_{\substack{V_{k}\subset H^{1}_{\partial B_r}(\Omega)}} \,\, \sup_{0\neq v\in V_{k}} \frac{\ds \int_{\Omega}\abs{\nabla v}^2\,dx}{\ds \int_{\partial \Omega_0} v^2\,d\mathcal{H}^{n-1}},
\end{equation}
where $V_{k}$ are the $k$-dimensional subspaces of $H^1_{\partial B_r }(\Omega)$, which is the space of $H^1(\Omega)$ functions that have null trace on $\partial \Omega$ (see the precise definition in subsection \eqref{subsec:SD}). We will denote by $m_k$ the multiplicity of the $k$-th eigenvalue $\sigma_k$.\\

When $r=0$, that is $\Omega=\Omega_0$ has no holes, we recover the classical Steklov eigenvalue problem in $\Omega_0$, i.e.
\begin{equation}\label{proS}
\begin{cases}
\Delta u=0 & \mbox{in}\ \Omega_0\vspace{0.2cm}\\
\dfrac{\de u}{\de \nu}=\bar{\sigma} u&\mbox{on}\ \partial\Omega_0\vspace{0.2cm}. 
\end{cases}
\end{equation}
Shape optimization problems concerning the first non-trivial eigenvalue trace back to R. Weinstock, who proved in \cite{Wein1954} an isoperimetric inequality in dimensions two. More precisely, he showed that the disc maximizes the first non-trivial Steklov-Laplacian eigenvalue in the class of planar simply connected sets with a perimeter constraint. Sixty years later, the result has been generalized in any dimension in \cite{BFNT2021}, where the authors proved the Weinstock inequality, restricting the class of admissible sets to the class of convex sets with fixed perimeter.
Regarding the measure costraint, in \cite{Bro2001} it is proved that the ball is still a maximizer for the first non-trivial Steklov eigenvalue among all bounded open sets with Lipschitz boundary.\\

Regarding the first Steklov-Dirichlet eigenvalue, several studies have focused on shape optimization issues. When $\Omega_0$ is a ball, in \cite{ftouh2022place, verma2018bounds} it is proved that $\sigma_1(\Omega)$ achieves the maximum when the two balls are concentric, i.e. when $\Omega$ is the spherical
shell. The same result has been proved in \cite{hong2020shape} by using a shape derivative approach and then it has been extended in the class of convex centrally symmetric sets in \cite{gavitone2024monotonicity}. Up to our knowledge, it is known that the spherical shell is a maximizer only if $\Omega_0$ is in suitable
classes of sets, when the volume of $\Omega_0$ and the radius of the inner ball are fixed. In \cite{PaPiSa} it is shown the maximality in the class of nearly spherical sets. In \cite{gavitone2021isoperimetric} an existence result in the class of convex sets is given with the proof of upper and lower bounds for $\sigma_1(\Omega)$ in terms of the minimum and maximum distance of the origin from the boundary $\partial \Omega_0$. Moreover, regarding the maximality, it is proved that the spherical shell is the optimal shape in the class of open, bounded, convex sets $\Omega_0$ with
Lipschitz boundary only if $\Omega_0$ is confined in a suitable ball (whose radius depends on the dimension of the space), when both the volume and the radius of inner ball are fixed. Lately, we stress that, in \cite{GS2023}, the Dirichlet boundary condition has been replaced by a Robin boundary condition on the inner ball: here the main properties of the first Steklov-Robin eigenvalue and the associated eigenfunctions are studied, with particular attention to their asymptotic behaviour when the Robin parameter goes to zero and infinity.  Moreover, for analogous problems with different boundary conditions on the connected components of $\partial \Omega$, see for instance \cite{BS2025,BV2024,cito2025,cito2024, DP2020,GP2024,PPT}.\\

The first result of this paper is related to the first Steklov-Dirichlet eigenvalue. In \cite{gavitone2021isoperimetric} the authors prove an isoperimetric inequality involving the first eigenvalue, when $\Omega_0$ is a suitable convex set: to be more precise, $\Omega_0$, must be contained in a ball $B_{R(n,r)}$, where the radius $R(n,r)>0$ depends on the dimension of the space and on the radius of the hole $B_r$. Our first Theorem allows us to prove an isoperimetric inequality when the radius of the hole $B_r$ is small enough. 
\begin{thm}\label{cor:isoperimetricsmallradius1}
 Let $\Omega_r=\Omega_0\setminus\overline{B}_r$, where $\Omega_0$ is an open, bounded and connected set in $\mathbb R^n$ with Lipschitz boundary and $B_r$ a ball of radius $r>0$ centered at the origin, such that $\overline{B}_r \subset \Omega_0$. There exists a $r_1^M=r_1^M(\Omega_0)$ depending only on the geometry of $\Omega_0$, such that for every $r\in (0,r^M_1)$, we get
   \begin{equation}\label{eq:isopfirstSDmeas}
       \sigma_1(\Omega_r) \le \sigma_1(A_{r,R_M}),
   \end{equation}
   where $R_M>0$ is the radius of the ball with the same measure as $\Omega_0$. The equality case holds if and only if $\Omega_0=B_{R_M}$.
   \end{thm}

We highlight that, differently from the result contained in \cite{gavitone2021isoperimetric}, here we manage to prove inequality \eqref{eq:isopfirstSDmeas} when $\Omega_0$ lives in a larger class of sets, also dropping its confinement into a suitable ball. We stress that if $\Omega_0$ is in  the class of open, bounded and connected sets with Lipschitz boundary, then $\sigma_1(\Omega_0)$ is well defined and always bounded above when fixing the measure of $\Omega_0$ and the radius $r>0$ (see \cite[Proposition $2.2$ and $2.4$]{paoli2020stability}).    The main ingredient of the proof is the convergence in $H^1(\Omega_0)$ of the suitably normalized eigenfunctions corresponding to $\sigma_1(\Omega_r)$ to particular constants (see proposition \ref{prop:firsteig}).\\

The second main result of this paper is related to the second Steklov-Dirichlet eigenvalue. It all started with a recent paper  \cite{BCV2023}, where the authors generalized the upper and lower bounds, contained in \cite{gavitone2021isoperimetric}, for doubly connected star-shaped domains in non-Euclidean
space forms, and obtained an isoperimetric inequality involving the $k$-th Steklov-Dirichlet eigenvalue, for $2\le k\le n+1$, under suitable geometrical assumption and with a measure constraint. To be more precise, they restricted their study in the class of open, bounded sets in $\mathbb R^n$, which are symmetric of order $4$, i.e., which are invariant under rotations of an angle of $\pi/2$, in the $\{x_i,x_j\}$ plane, for every $i,j\in\{1,...,n\}$. What they proved is the following
\begin{thm}[S. Basak, A. Chorwadwala, S. Verma]
    Let $\Omega_0$ be an open, bounded set in $\mathbb R^n$, symmetric of order $4$. Let $\Omega_r=\Omega_0\setminus\overline{B}_r$, where $B_r$ is a ball of radius $r>0$ centered at the origin, such that $\overline{B}_r \subset \Omega_0$. Then, for every $2\le k\le n+1$
   \begin{equation}\label{eq:isopVerma}
       \sigma_k(\Omega_r) \le \sigma_k(A_{r,R})=\sigma_2(A_{r,R}),
   \end{equation}
   where $A_{r,R}$ is the spherical shell having the same measure as $\Omega$.
\end{thm}
Furthermore, the authors provided a counterexample for which inequality \eqref{eq:isopVerma} may not hold if the symmetry condition is dropped (see  \cite[Remark 5.1]{BCV2023}). In particular they proved, in the planar case, that the second Steklov-Dirichlet eigenvalue of a non-centrally symmetric perforated ellipse is strictly greater than the one of the annulus having its same measure.\\

Nevertheless, this counterexample does not take into count the fact that the hole of the set is ``\textit{too big}". Indeed the second main goal of this paper is to prove inequality \eqref{eq:isopVerma} for $k=2$, when the radius of the hole is small enough, and $\Omega_0$ is an open, bounded set with Lipschitz boundary with fixed measure, or when $\Omega_0$ is an open, bounded and convex set, with fixed perimeter.\\

To this aim, we need to study the connection between problems \eqref{proSD} and \eqref{proS}, which is suggested by the following simple considerations. With Proposition \ref{prop:firsteig}, we will see that the convergence of the Steklov-Dirichlet eigenvalue  and corresponding eigenfunctions to the ones of the trivial Steklov eigenvalue in the non-perforated domain. In addition, a direct computation in the radial case shows that the second Steklov-Dirichlet eigenvalue tends to the first non-trivial Steklov eigenvalue, when the radius of the hole vanishes, and the same happens for the corresponding eigenfunctions in a pointwise convergence (see proposition \ref{prop:convergenceeigrad}). 

Hence, if we denote by $\{\sigma_k\}_{k\in \mathbb N}$ and by $\{\bar{\sigma}_k\}_{k\in \mathbb N_0}$ respectively the spectra of the problem \eqref{proSD} and \eqref{proS}, one might ask if a similar behaviour is observed for $\sigma_2(\Omega)$ and $\overline{\sigma}_1(\Omega_0)$ and the corresponding eigenfunctions.
Our second main Theorem answers positively to these questions. 
\begin{thm}\label{thm:main}
Let $\Omega=\Omega_0\setminus\overline{B}_r$, where $\Omega_0$ is an open, bounded and connected set with Lipschitz boundary, and $B_r$ a ball of radius $r>0$ centered at the origin, such that $\overline{B}_r \subset \Omega_0$. Then we have
\begin{equation*}
    \lim_{r\to 0^+}\sigma_{2}(\Omega_r) = \bar{\sigma}_{1}(\Omega_0).
\end{equation*}
Moreover, there exists a sequence of eigenfunctions $\{u_2^r\}_r$ corresponding to $\sigma_2(\Omega_r)$ and an eigenfunction $\overline{u}$ corresponding to $\bar{\sigma}_{1}(\Omega_0)$, such that $u^2_r$ converges to $\overline{u}$ in $H^1(\Omega_0)$, as $r\to 0^+$.
\end{thm}

The proof relies on a suitable choice of test functions, with particular attention to the rate of convergence of the radius of the inner ball.\\

\noindent As a consequence of Theorem,
we obtain two isoperimetric inequalities. The first is proved using the Brock inequality.
\begin{cor}\label{thm:isoperimetric1}
   Let $\Omega_r=\Omega_0\setminus\overline{B}_r$, where $\Omega_0$ is an open, bounded and connected set in $\mathbb R^n$ with Lipschitz boundary and $B_r$ a ball of radius $r>0$ centered at the origin, such that $\overline{B}_r \subset \Omega_0$. Then, there exists a $r^M_2=r^M_2(\Omega_0)$ that only depends on the geometry of $\Omega_0$, such that for every $r\in (0,r^M_2)$, we get
   \begin{equation}\label{eq:isopsecondSD}
       \sigma_2(\Omega_r) \le \sigma_2(A_{r,R_M}),
   \end{equation}
   where $R_M>0$ is the radius of the ball with the same measure as $\Omega_0$.
\end{cor}
The second one is proved by using the Weinstock inequality instead.
\begin{cor}\label{thm:isoperimetric2}
   Let $\Omega_r=\Omega_0\setminus\overline{B}_r$, where $\Omega_0$ is an open, bounded and convex set and $B_r$ a ball of radius $r>0$ centered at the origin, such that $\overline{B}_r \subset \Omega_0$. There exists a $r^P_2=r^P_2(\Omega_0)$ that only depends on the geometry of $\Omega_0$, such that for every $r\in (0,r^P_2)$, we get
   \begin{equation}\label{eq:isopsecondSDP}
       \sigma_2(\Omega_r) \le \sigma_2(A_{r,R_P}),
   \end{equation}
   where $R_P>0$ is the radius of the ball with the same perimeter as $\Omega_0$.
\end{cor}

\begin{rem}
We stress that, in dimension two, Corollary \ref{thm:isoperimetric2} holds in the class of planar simply connected domains.
\end{rem}

\noindent \textbf{Plan of the paper}: In Section \ref{sec:2} we write down some notation and some basic and new preliminary result:  we analyze the radial case;  we study the number of nodal domains of the eigenfunctions corresponding to the $k$-th Steklov-Dirichlet eigenvalue, proving in particular that for $k=2$, the nodal domains are exactly $2$; in Section \ref{sec:3} we prove the main results: firstly Theorem \ref{cor:isoperimetricsmallradius1} regarding the first Steklov-Dirichlet eigenvalue; then we give the proof of Theorem \ref{thm:main} and Corollaries \ref{thm:isoperimetric1}-\ref{thm:isoperimetric2}, about the second Steklov-Dirichlet eigenvalue.

\section{Preliminaries}\label{sec:2}
In this section we will give some notation and definition. Moreover we will prove some classical result adapted to the perforated case (such as Corollary \ref{cor:nodaldomains2}) and some new result, as the ones in \ref{prop:convergenceeigrad}-\ref{prop:ballnotnodal}. \subsection{Notations}
Throughout this paper, we denote by $B_r(x_0)$ and $B_r$ the balls in $\mathbb{R}^n$ of radius $r>0$ centered at  $x_0\in\mathbb{R}^n$ and at the origin, respectively. Moreover with $B$, $\mathbb{S}^{n-1}$ and $\omega_n$ we will denote respectively the unit ball of $\mathbb R^{n}$, its boundary and its volume. Let $r,R$ be such that $0<r<R$, the spherical shell will be denoted as follows:
\begin{equation*}
A_{r,R}=\{x \in \mathbb R^{n} \colon\, r<|x|<R\}.
\end{equation*}
The Lebesgue measure of a measurable set $E \subset \mathbb R^{n}$ will be denoted by $|E|$. Moreover, the $(n-1)$-dimensional Hausdorff measure in $\mathbb{R}^n$ will be denoted by $\mathcal H^{n-1}$ and the Euclidean scalar product in $\mathbb{R}^n$ is denoted by $\langle\cdot,\cdot\rangle$. Since we will always work with sets that have at least Lipschitz boundary, then we will denote by $P(E)=\mathcal{H}^{n-1}(\partial E)$ the perimeter of $E$.

\subsection{Friedrich inequalities}
Let $\Omega$ be an open bounded subset of $\mathbb{R}^n$ with Lipschitz boundary. Besides the classical Sobolev trace inequality (see for instance \cite{evans2010partial}) that gives us the compactness of the embedding operator of $H^1(\Omega)$ into $L^2(\partial \Omega)$, we want to introduce another important embedding Theorem, which is a consequence of the so-called  Friedrich's inequality (see for instance \cite{friedrichs1928randwert, maz2013sobolev} and for a more general case \cite{cianchi2016sobolev}). 
Let $H^1(\Omega,\partial \Omega)$ the completion of the set of functions in $C^{\infty}(\Omega)\cap C(\Bar{\Omega})$ which have weak gradient in $L^2(\Omega)$, equipped with the following norm (see \cite{maz2013sobolev} for the details)
\begin{equation*}
\|u\|_{H^1(\Omega,\partial \Omega)}=\|\nabla u\|_{L^2(\Omega)}+\|u\|_{L^2(\partial\Omega)}.
\end{equation*}
Friedrich's inequality states that
\begin{equation}\label{friedin}
    \|u\|_{L^2(\Omega)}\le C(\|\nabla u\|_{L^2(\Omega)}+\|u\|_{L^2(\partial\Omega)})
\end{equation}
for some positive constant $C>0$. Also in this case the embedding operator of $H^1(\Omega,\partial \Omega)$ into $L^2(\Omega)$ is compact (see Corollary 3, p. 392 in \cite{maz2013sobolev}).

\subsection{On the Steklov eigenvalue problem}
Let $\Omega_0 \subset 
 \mathbb R^n$ be a simply connected open, bounded set with Lipschitz boundary and let us consider problem \eqref{proS}. We will say that a real number $\overline{\sigma}$ and a $H^1(\Omega_0)$ function $v$ are called eigenvalue and its corresponding eigenfunction to problem \eqref{proS} if and only if
\begin{equation}\label{eq:WeakSteklov}
\int_{\Omega_0}\nabla \varphi \cdot \nabla v \,dx = \overline{\sigma}\int_{\partial \Omega_0}\varphi v \,d\mathcal{H}^{n-1}, \qquad \forall \varphi \in H^1(\Omega_0).
\end{equation}
It is well known that the Steklov spectrum is formed by a sequence of non-negative diverging real numbers, where the first eigenvalue $\overline{\sigma}_0$ is trivial and the corresponding eigefunctions are the costants. In particular the first non-trivial Steklov eigenvalue have the following variational characterization 
\begin{equation*}
    \overline{\sigma}_1(\Omega_0)=  \inf_{0\neq v \in H^1(\Omega_0)}\bigg\{\frac{\int_{\Omega_0}\abs{\nabla v}^2\,dx}{\int_{\partial \Omega_0}v^2\,d\mathcal{H}^{n-1}} : \int_{\partial \Omega_0}v\,d\mathcal{H}^{n-1}=0 \bigg\}.
\end{equation*}
As already mentioned in the introduction, $\overline{\sigma}_1$ satisfies some isoperimetric inequality. To be more precise F. Brock (see \cite{Bro2001}) proved the following scaling invariant inequality
\begin{equation}\label{eq:brock}
    \overline{\sigma}_1(\Omega_0)\abs{\Omega_0}^\frac{1}{n}\le \overline{\sigma}_1(B_R)\abs{B_R}^\frac{1}{n},
\end{equation} 
in the class of open, bounded sets with Lipschitz boundary, and R. Weinstock (see \cite{Wein1954}) proved, in dimension $2$ and in the class of simply connected sets with Lipschitz boundary, the following isoperimetric inequality
\begin{equation}\label{eq:weinstock}
    \overline{\sigma}_1(\Omega_0)P(\Omega_0)^\frac{1}{n-1}\le \overline{\sigma}_1(B_R)P(B_R)^\frac{1}{n-1},
\end{equation} 
successively generalized in every dimension, in the class of open, bounded and convex sets in \cite{BFNT2021}.

\subsection{On the Steklov-Dirichlet eigenvalue problem}\label{subsec:SD} Let $\Omega_0$ be an open, bounded and connected set with Lipschitz boundary and let us consider a ball $B_r$ of radius $r>0$ centered at the origin, such that $\overline{B}_r \subset \Omega_0$. Let us denote by
\begin{equation*}
C^\infty_{\partial B_{R_1}} (\Omega):=\{ u_{|\Omega}  \ | \ u \in C_0^\infty (\R^n),\ \spt (u)\cap \partial B_{R_1}=\emptyset \}.  
\end{equation*}
Then we define the space $H^1_{\partial B_{R_1}}(\Omega)$ as follows
\begin{equation*}
    H^1_{\partial B_{R_1}}(\Omega):= \overline{C^\infty_{\partial B_{R_1}} (\Omega)}^{H^1(\Omega)}.
\end{equation*}
Let us now consider problem \eqref{proSD}. We will say that a real number $\sigma$ and a $H^1_{\partial B_r}(\Omega)$ function $v$ are an eigenvalue and its corresponding eigenfunction to problem \eqref{proS} if and only if
\begin{equation}\label{eq:WeakSteklovdirichlet}
\int_{\Omega_0}\nabla \varphi \cdot \nabla v \,dx = \sigma\int_{\partial \Omega_0}\varphi v \,d\mathcal{H}^{n-1}, \qquad \forall \varphi \in H^1_{\partial B_r}(\Omega).
\end{equation}

In \cite{PaPiSa}, it is proved that the first Steklov-Dirichlet eigenvalue, which has the following variational characterization
\begin{equation*}
    \sigma_1(\Omega) = \inf_{0\neq v \in H^1_{\partial B_r}(\Omega)}\frac{\int_{\Omega}\abs{\nabla v}^2\,dx}{\int_{\partial \Omega_0}v^2\,d\mathcal{H}^{n-1}},
\end{equation*}is positive, simple and consequently, the corresponding eigenfunctions cannot change sign in $\Omega$. Since the Neumann-to-Dirichlet operator is compact and self-adjoint in $H^{-\frac12}(\partial \Omega_0)$ (see \cite{Agr2006}), then we can compute the eigenvalues 
as follows (see for instance \cite[Chapter $1$]{weinstein1972methods}. Let use denote by $W_j=\{v\in H^1_{\partial B_r}(\Omega)\; : \; \int_{\partial\Omega_0}vu_i\,d\mathcal{H}^{n-1}=0,\; i=1,...,j-1\}$, where $u_i$ are eigenfunctions corresponding to $\sigma_i(\Omega)$.  Then for $k\ge 2$
\begin{equation*}
    \sigma_k(\Omega) =\inf_{0\neq v \in W_k}\frac{\int_{\Omega}\abs{\nabla v}^2\,dx}{\int_{\partial \Omega_0}v^2\,d\mathcal{H}^{n-1}}.
\end{equation*}

In particular we can write the $2$nd Steklov-Dirichlet eigenvalue as follows. Let $u_1$ be an eigenfunction corresponding to $\sigma_1(\Omega)$. Then
\begin{equation}\label{eq:varcars2}
    \sigma_2(\Omega) = \inf_{0\neq v \in H^1_{\partial B_r}(\Omega)}\bigg\{\frac{\int_{\Omega}\abs{\nabla v}^2\,dx}{\int_{\partial \Omega_0}v^2\,d\mathcal{H}^{n-1}} : \int_{\partial \Omega_0}u_1v\,d\mathcal{H}^{n-1}=0\bigg\}.
\end{equation} 

\begin{rem}
    We stress that when we fix the measure or the perimeter of $\Omega_0$ and the radius $r>0$ of the inner ball $B_r$ is fixed and small enough, then $\sigma_2(\Omega)$ is always bounded above in the class of open, bounded and connected sets with Lipschitz boundary (see Step $2$ in the proof of \ref{thm:main}). In particular the $\inf$ in \eqref{eq:varcars2} is actually a minimum and the proof of the existence follows the same lines of \cite[Proposition $2.2$]{paoli2020stability}, with just the additional carefulness of the orthogonality condition that must be satisfied by the minimizing sequence. But this is not a real issue, since a minimizing sequence will converge strongly in $L^2(\partial \Omega_0)$ to minimum.
\end{rem}

\subsubsection{\textbf{The radial case}}\label{sec:2radial}
Let $\Omega_0= B_R$ be a ball centered at the origin with radius $R>r>0$, and let us consider problem \eqref{proSD} in $\Omega=A_{r,R}$.
In this case, we know that the first Steklov-Dirichlet eigenvalue and its corresponding eigenfunctions (which are radial) can be computed and are given by
\begin{equation}\label{eq:radeigenvf}
\sigma_1(A_{r,R})=
\begin{cases}
\frac{1}{R\log\left(\frac{R}{r}\right)}& {\rm for}\;\; n=2\vspace{0.1cm}\\
\frac{n-2}{R\left[\left(\frac{R}{r}\right)^{n-2}-1\right]}& {\rm for}\;\; n\geq 3\vspace{0.1cm},\\ 
\end{cases}\qquad 
w(s)=\begin{cases}
\ln s-\ln r
& {\rm for}\;\; n=2\vspace{0.1cm}\\
\left( \dfrac{1}{r^{n-2}}-\dfrac{1}{s^{n-2}}\right)& {\rm for}\;\; n\geq 3\vspace{0.1cm},
\end{cases}
\end{equation}
with $s=|x|$. Regarding the explicit expression of the eigenvalues $\sigma_k(A_{r,R})$, for $k\ge 2$, and their corresponding eigenfunctions we refer to \cite{BCV2023}. In the following proposition we will write them only in case $k=2$. 

\begin{prop}
Let $\Omega=A_{r,R}$. The second Steklov-Dirichlet eigenvalue of problem \eqref{proSD} has multiplicity $n$ and it is given by
\begin{equation} \label{secondeigrad}
    \sigma_2(A_{r,R}) = \frac{R^n +r^n (n-1)}{R[R^n-r^n]}.
\end{equation}
Moreover the corresponding eigenspace is spanned by the following eigenfunctions (in polar coordinates) 
\begin{equation}\label{eq:radeig}
    w_j(s,\theta) = \bigg(s-\frac{r^n}{s^{n-1}}\bigg)Y^1_j(\theta), \qquad j=1,\dots, n,
\end{equation}
where $(s,\theta)\in [r,R] \times \mathbb{S}^{n-1}$ and $Y^1_j(\theta)$ are the eigenfunctions corresponding to the first eigenvalue of the Laplace Beltrami operator on $\mathbb S^{n-1}$.
\end{prop}

 It is straightforward to check that in any dimension $\lim_{r\to 0^+}\sigma_1(A_{r,R}) = 0$, where the zero value can be interpreted as the first trivial Steklov eigenvalue in $B_R$. Moreover, it is easy to prove that the  second Steklov-Dirichlet eigenvalue tends to the first non-trivial Steklov eigenvalue when the radius of the hole goes to zero, and the same happens for the corresponding eigenfunctions.

\begin{prop}\label{prop:convergenceeigrad}
The second Steklov-Dirichlet eigenvalue in the spherical shell $A_{r,R}$ is strictly increasing with respect to $r\in [0,R)$. In particular, it tends to the first non-trivial Steklov eigenvalue in $B_R$ as $r\to 0^+$, i.e.
\begin{equation*}
    \lim_{r\to 0^+} \sigma_2(A_{r,R}) = \overline{\sigma}_1(B_R).
\end{equation*}
Moreover the corresponding eigenfunctions in \eqref{eq:radeig} converges pointwise to the eigeinfunctions corresponding to $\overline{\sigma}_1(B_R)$ as $r\to 0^+$.
\end{prop}
\begin{proof}
    If we differentiate \eqref{secondeigrad} with respect to $r$, we get
    \begin{equation*}
        \frac{d}{dr}\sigma_2(A_{r,R}) = \frac{n^2R^{n-1}r^{n-1}}{(R^n-r^n)^2}> 0, \qquad r\in (0,R),
    \end{equation*}
    and it is $0$ if and only if $r=0$. Passing to the limit as $r\to 0^+$ we get
    \begin{equation*}
        \sigma_2(A_{r,R}) \to \frac{1}{R} = \overline{\sigma}_1(B_R),
    \end{equation*}
    and in particular we get
    \begin{equation}\label{eq:SleSD}
        \overline{\sigma}_1(B_R)\le \sigma_2(A_{r,R}), \qquad \forall r\in(0,R).
    \end{equation}

Let us now show that the eigenfunctions corresponding to the second Steklov-Dirichlet eigenvalue in $A_{r,R}$ tend to the eigenfunctions corresponding to the first non-trivial Steklov eigenvalue on $B_R$ as $r\to 0^+$. 
Since $Y_j^1(\cdot)$ can be written in cartesian coordinates as
\begin{equation*}
    Y^1_j(x)=\frac{x_j}{\abs{x}} \qquad j=1,...,n,
\end{equation*}
we have that
\begin{equation}\label{eq:radialeigsigma2}
    w_j(x) = \bigg(|x|-\frac{r^n}{|x|^{n-1}}\bigg)\frac{x_j}{|x|}=  \bigg(1-\frac{r^n}{|x|^n}\bigg)x_j, \qquad j=1,...,n.
\end{equation}
In particular we get
\begin{equation*}
    \lim_{r\to 0^+}w_j(x)= x_j, \qquad j=1,...n,
\end{equation*}
which are exactly the eigenfunctions corresponding to the first non-trivial Steklov eigenvalue on the ball $B_R$.
\end{proof}

\subsubsection{\textbf{About the nodal domains}}\label{sec:4}
In this subsection we study what is the number of nodal domains of the eigenfunctions corresponding to the $k-$th Steklov-Dirichlet eigenvalue. Let $\Omega= \Omega_0\setminus\overline{B_r}$ as usual and let us denote by $V_k$ the eigenspace corresponding to $\sigma_k(\Omega)$. 

\begin{defn}
    We will call $\Omega_j$ a nodal domain of $u\in V_k$ in $\Omega$, a connected component of the set $\{x\in \Omega: u(x)\neq 0\}$. In particular we will denote by
    \begin{equation}
        \Delta_k = \max_{v \in V_k} \sharp \{\text{nodal domains of v} \}
    \end{equation}
\end{defn}
Firstly we prove $B_r$ is not well contained in any of the nodal domains, for every $k\ge 2$.
\begin{prop}\label{prop:ballnotnodal}
    Let $k\ge 2$ and let us consider $0\neq u\in V_k$ and let $\Omega_j$ be the nodal domains of $u$ for $j=1,...,m$, for some $m\in \mathbb N$. $\nexists j \in \{1,...,m\}$ such that $B_r\subset\subset \Omega_j$.
\end{prop}
\begin{proof}
    By contradiction, let us suppose that $B_r \subset\subset \Omega_j$ for some $j$. We know that $u$ is a harmonic function in $\Omega_j$ that is zero in $\partial B_r$ and $\partial \Omega_j \cap \Omega$. In $\Omega_j$, we have $u>0$ or $u<0$ and for simplicity we will consider only case $u>0$, since the other case follows analogously. Since $u>0$, its minimum is $0$ and the maximum is achieved on $\partial\Omega_j\cap\partial\Omega_0$, but since $u$ is smooth and zero on the two disconnected boundary components $\partial B_r$ and $\partial \Omega_j \cap \Omega$, then it follows that  there exists a relative maximum $x_0$ in $\Omega_j$. This implies that $x_0$ is a global maximum in an open neighbourhood $U_{x_0}$ of $x_0$. Since $u|_{U_{x_0}}$ is harmonic, by the maximum principle it follows that $u$ is constant in $U_{x_0}$. Since harmonic functions are analytic functions as well, this implies $u$ constant in $\Omega_j$ and consequently $u$ is zero in $\Omega_j$. This is a contradiction since it goes against the definition of nodal domain.  
\end{proof}
The previous proposition allows us to define the \textit{boundary nodal domains} of $u\in V_k$ as the set $(\partial \Omega_j \cap \overline{\Omega})\setminus \partial \Omega_0$. The next
proposition is general and gives an upper bound for the number of nodal domains that an eigenfunction corresponding to the $k$-th eigenvalue can have. The arguments follow exactly the same lines contained in \cite{AM1994,KS1969}, with little suitable modifications, since the presence of a hole in $\Omega_0$. 

\begin{prop}\label{prop:nodaldomains}
    For any $l\ge 1$ we have the following inequality
    \begin{equation*}
        \Delta_{l+1}\le 1+ \sum_{j=1}^l m_j,
    \end{equation*}
where $m_j$ is the multiplicity of the $j-$th eigenvalue.
\end{prop}
In the previous Subsection we have seen what is the explicit expression \eqref{eq:radialeigsigma2} of the eigenfunctions corresponding to the second Steklov-Dirichlet eigenvalue on the spherical shell. In particular, by \eqref{eq:radialeigsigma2} it is straightforward to see that the nodal domains are exactly $2$. Now, as a consequence of Proposition \ref{prop:nodaldomains}, we want to prove that it holds true for any $\Omega$.
\begin{cor}\label{cor:nodaldomains2}
    The eigenfunctions associated to $\sigma_2(\Omega)$ have only two nodal domains.
\end{cor}
\begin{proof}
    Let $v \in V_2$ be an eigenfunction associated to $\sigma_2 (\Omega)$. In \cite{GS2023} it has been proved in the Steklov-Robin case (and so even in the Steklov-Dirichlet case when the Robin parameter goes to $+\infty$)  that an eigenfunction does not change sign if and only if it is a first eigenfunction. This implies that $v$ is sign changing, which means that $\Delta_2 \ge 2$. Applying Proposition \ref{prop:nodaldomains}, we have \begin{equation*} \Delta_2 = 2.
    \end{equation*}
\end{proof}

\subsection{A corrector function} In this subsection we consider a particular function (which depends on the dimension of the space) that will play a central role in proving the main results, that was introduced by \cite{CM1997}.
Let $\varepsilon>0$ be small enough and let us consider $\Omega_{r_\varepsilon}= \Omega_0\setminus B_{r_\varepsilon}$, where $B_{r_\varepsilon}$ is a ball centered at the origin with radius $0<r_\varepsilon<\varepsilon$, such that $\overline{B}_{r_\varepsilon}\subset \Omega_0$. We define the following corrector function in $H^1(\Omega_0)$
\begin{equation}\label{eq:omegaeps}
    \omega^\varepsilon(x) = \begin{cases}
        0 & \,\text{in}\,B_{r_\varepsilon}\\ 
        \rho_n(\abs{x})& \;\text{in}\; A_{r_{\varepsilon},\varepsilon}\\ 
        1 & \,\text{in}\, \Omega_0\setminus B_\varepsilon,
    \end{cases}
\end{equation}
where $\rho_n(x)$ is the fundamental solution of the Laplacian in $A_{r_{\varepsilon},\varepsilon}$, i.e. solves the following PDE
\begin{equation*}
    \begin{cases}
        \Delta w=0  & \;\text{in}\;A_{r_{\varepsilon},\varepsilon}\\
        w=0 & \;\text{on}\; \partial B_{r_\varepsilon}\\
        w=1 & \;\text{on}\; \partial B_\varepsilon.
    \end{cases}
\end{equation*}
In particular, it can be explicitly computed and it is given by
\begin{equation*}
       \rho_n(\abs{x})= \begin{cases}
\displaystyle\frac{\log(\frac{\abs{x}}{r_\varepsilon})}{\log(\frac{\varepsilon}{r_\varepsilon})} & \, n=2\\ \\
\displaystyle\frac{r_\varepsilon^{-n+2}-\abs{x}^{-n+2}}{r_\varepsilon^{-n+2}-\varepsilon^{-n+2}} & \, n\ge 3.
        \end{cases}
\end{equation*}
The following lemma will be a key point to prove the second  main result of the paper (see \cite{CM1997}).
\begin{lem}\label{lemma:omega}
    Let $\omega^\varepsilon$ defined as above. Let us assume that
    \begin{equation}\label{eq:r=o(eps)}
    \begin{cases}
        r_\varepsilon = o(\varepsilon) & n=2\\
        r_\varepsilon = o(\varepsilon^{\frac{n}{n-2}}) & n\ge 3.
    \end{cases}
        \end{equation}
    Then we have 
    \begin{equation}\label{eq:normomega}
        \lim_{\varepsilon\to 0^+}\|\omega^\varepsilon\|_{L^2(\Omega_0)} = |\Omega_0|,
    \end{equation}
which is equivalent to $\omega^\varepsilon\to 1$ for a.e. $x\in\Omega$. Moreover \begin{equation}\label{eq:normnablaomega}
    \lim_{\varepsilon\to 0^+}\|\nabla \omega^\varepsilon\|_{L^2(\Omega_0)} = 0.
\end{equation}
\end{lem}

\section{The main results}\label{sec:3}

Now we are able to give the proof of the main result of the paper.
\subsection{Results related to $\sigma_1$ and the corresponding eigenfunctions}
The first step is to prove that the normalized eigenfunctions corresponding to $\Omega_r= \Omega_0\setminus\overline{B}_r$ converge in $H^1(\Omega_0)$ (we assume that the eigenfunctions are zero in $B_r$) to a particular constant.
\begin{prop}\label{prop:firsteig}
    Let $\Omega=\Omega_0\setminus\overline{B}_r$, where $\Omega_0$ is an open, bounded and connected set with Lipschitz boundary, and $B_r$ a ball of radius $r>0$ centered at the origin, such that $\overline{B}_r \subset \Omega_0$. If we denote by $\{u_1^{r}\}$ the positive eigenfunctions corresponding to $\sigma_1(\Omega_r)$, such that $\|u_1^{r}\|_{L^2(\partial\Omega_0)}=1$, then we have
    \begin{equation}
        u_1^{r} \to c_{\Omega_0} = \frac{1}{\sqrt{P(\Omega_0)}} \qquad  \text{in}\;\;H^1(\Omega_0),\; \text{as}\;\; r\to0^+.
    \end{equation}
    \end{prop}
\begin{proof}
    Firstly, we recall that $\sigma_1(\Omega_r)\to 0$ as $r\to 0^+$ as it was already proved in \cite{gavitone2021isoperimetric}. Here we briefly recall the proof as it easily follows by the monotonicity of the eigenvalues with respect to the set inclusion. Indeed, if we consider the annulus $A_{{r},R_m}$, where $R_m$ is the minimum distance of the origin from the boundary of $\Omega_0$, then $A_{{r},R_m}\subset \Omega_{r}$ and we have
    \begin{equation}\label{eq:sigmatozero}
        0<\sigma_1(\Omega_{r})\le \sigma_1(A_{{r},R_m}) \to 0, \qquad \text{as}\; r\to 0^+,
    \end{equation}
    where we recall that the explicit expression of $\sigma_1(A_{{r},R_m})$ is written in \eqref{eq:radeigenvf}.
    By the variational characterization of the first eigenvalue, the normalization constraint and \eqref{eq:sigmatozero}, we have that $\|\nabla u_1^{r}\|_{L^2(\Omega_0)}\le 1 $ for every ${r}$ small enough. Moreover, extending to $0$ in $B_{r}$, by the Friederich's inequality \eqref{friedin} we have that the sequence $\{u_1^{r}\}_e$ is equibounded in $L^2(\Omega_0)$. Therefore, there exists a subsequence, still denoted by $\{u_1^{r}\}_r$ that converges strongly in $L^2(\Omega_0)$ to some $\Bar{u}\in H^1(\Omega_0)$ and such that $\nabla u_1^{r} \rightharpoonup \nabla \Bar{u}$ weakly in $L^2(\Omega_0)$. Actually, the convergence of the gradients is strong, since
    \begin{equation*}
        0\le \|\nabla u^{r}_1\|^2_{L^2(\Omega_0)}= \sigma(\Omega_{r})\to 0 \qquad \text{as}\; r \to 0^+.
    \end{equation*}
    Moreover, by the compactness of the trace operator we have that $u_1^{r}$ converges strongly in $L^2(\partial \Omega_0)$ and almost everywhere in $\partial \Omega_0$. Therefore, it must be
    \begin{equation*}
        \Bar{u}= c_{\Omega_0} \in (0,+\infty).
    \end{equation*}
    In particular, the normalization condition yields
    \begin{equation*}
        c_{\Omega_0} = \frac{1}{\sqrt{P(\Omega_0)}}.
    \end{equation*}
    Since the limit does not depend on the choice of the subsequence, we have the strong convergence in $H^1(\Omega_0)$ of $u_1^{r}$ to $\bar{u}$ and this concludes the proof.

\end{proof}

\begin{rem}\label{rem_constants}
    Let $A_{r,R_M}$ be the spherical shell having the same measure as $\Omega$. Then we can compare the two constants $c_\Omega$ and $c_{B_{R_M}}$. Indeed, since $\abs{\Omega_0}=\abs{B_{R_M}}$, by the isoperimetric inequality
    \begin{equation}\label{eq:isoperimetricconstantmeasure}
        c_{\Omega_0} = \frac{1}{\sqrt{P(\Omega_0)}} \le \frac{1}{\sqrt{P(B_R)}}= c_{B_{R_M}},
    \end{equation}
    and the equality case holds if and only if $\Omega_0=B_{R_M}$.
    We stress that a perimeter constraint would imply $c_{\Omega_0} = c_{B_{R_P}}.$
\end{rem}

\begin{lem}\label{rem:eigenfunctionradialindependence}
Let $\Omega_0$ be an open, bounded set with Lipschitz boundary in $\mathbb R^n$. Let $v_{r}$ be the eigenfunction corresponding to $\sigma_1(A_{r,\Tilde{R}})$, with $\Tilde{R}>0$ large enough so that $\Omega_0\subset\subset B_{\Tilde{R}}$. Moreover, let be $\|v_{r}\|_{L^2(\partial B_{R_M})}=1$, where $B_{R_M}$ is the ball having the same measure as $\Omega_0$. Then we have 
 \begin{equation}\label{eq:denominatorisoperimetric}        \int_{\partial\Omega_0} v_{r}^2\,d\mathcal{H}^{n-1}\to c^2_{B_{R_M}}P(\Omega_0), \qquad r\to 0^+.
    \end{equation}
\end{lem}
\begin{proof}

  Following the proof of Theorem \ref{prop:firsteig}, we get
    \begin{equation*}
        v_{r}\to c_{B_{R_M}}= \frac{1}{\sqrt{P(B_{R_M})}}\qquad \text{in}\;\; H^1(B_{\Tilde{R}}).
    \end{equation*}
    The thesis easily follows by an application of the trace inequality. Indeed 
    \begin{equation*}
        \|v_{r}-c_{B_{R_M}}\|_{L^2(\partial \Omega_0)}\le C \|v_{r}-c_{B_{R_M}}\|_{H^1(\Omega_0)}\le C \|v_{r}-c_{B_{R_M}}\|_{H^1(B_{\Tilde{R}})}
    \end{equation*}
    and the last term in goes to zero as $r$ goes to zero.
\end{proof}

Now we can prove the desired isoperimetric inequalitiy. 

\begin{proof}[Proof of Theorem \ref{cor:isoperimetricsmallradius1}] If $\Omega_0=B_R$ is a ball, there is nothing to prove. So, let us suppose that $\Omega_0\neq B_R$. Let $A_{r,{R_M}}$ be the spherical shell having the same measure as $\Omega$, and let us consider the eigenfunction $v_{r}$ corresponding to $\sigma_1(A_{r,{R_M}})$, such that $\|v_{r}\|_{L^2(\partial B_{R_M})}=1$. 
Now we can see $v_{r}$ as extended in the spherical shell $A_{r,\Tilde{R}}$ that contains $\Omega_{r}$ and for which \eqref{eq:denominatorisoperimetric} holds. Therefore, we can use $v_{r}$ as a test function in the variational characterization of $\sigma_1(\Omega_{r})$, having
\begin{equation}\label{eq:testfunccarvar}
    \sigma_1(\Omega_{r})\le \frac{\displaystyle\int_{\Omega_{r}}\abs{\nabla v_{r}
    }^2\,dx}{\displaystyle\int_{\partial\Omega_0} v_{r}^2\,d\mathcal{H}^{n-1}}.
\end{equation}
At the numerator we argue as in \cite[Theorem $1.1$]{gavitone2021isoperimetric}. Since the definition of $v_{r}$ in \eqref{eq:radeigenvf}, it is easy to check that $\abs{\nabla v_{r}}^2$ is a radial and decreasing function. Hence, by the Hardy-Littlewood inequality (see \cite{Kes2006}) we get
\begin{equation}\label{eq:numeratorcarvar}
    \int_{\Omega_{r}}\abs{\nabla v_{r}
    }^2\,dx\le\int_{A_{r,{R_M}}}\abs{\nabla v_{r}
    }^2\,dx.
\end{equation}
We stress that the equality case holds if and only if $\Omega_0=B_{R_M}$, therefore in this case we have a strict inequality.
    Regarding the denominator, we use \eqref{eq:denominatorisoperimetric} in Lemma \ref{rem:eigenfunctionradialindependence}.
    For every $\eta>0$, there exists a $r^M_1=r^M_1(\eta)>0$, such that for every $0<r< r^M_1$, then
    \begin{equation}\label{eq:limitdef}
        c^2_{B_{R_M}}P(\Omega_0)-\eta<\int_{\partial\Omega_0} v_{r}^2\,d\mathcal{H}^{n-1}.
    \end{equation}
    In particular, since $\Omega_0\neq B_{R_M}$, we can choose $\eta$ so that
    \begin{equation}\label{eq:choiceofeta}
        \eta = c^2_{B_{R_M}}( P(\Omega_0)-P(B_{R_M}))>0,
    \end{equation}
    which is strictly positive since the isoperimetric inequality holds. By this choice and inequality \eqref{eq:limitdef}, we get 
\begin{equation}\label{eq:denominatorcarvar1}
    \int_{\partial\Omega_0} v_{r}^2\,d\mathcal{H}^{n-1}>c^2_{B_{R_M}}P(\Omega_0)-\eta = c^2_{B_{R_M}}P(B_{R_M})=\int_{\partial B_{R_M}} v_{r}^2\,d\mathcal{H}^{n-1}.
\end{equation}
Eventually, putting \eqref{eq:numeratorcarvar}-\eqref{eq:denominatorcarvar1} in \eqref{eq:testfunccarvar}, we get
    \begin{equation*}
       \sigma_1(\Omega_{r})\le \frac{\displaystyle\int_{\Omega_{r}}\abs{\nabla v_{r}
    }^2\,dx}{\displaystyle\int_{\partial\Omega_0} v_{r}^2\,d\mathcal{H}^{n-1}}<\frac{\displaystyle\int_{A_{r,{R_M}}}\abs{\nabla v_{r}
    }^2\,dx}{\displaystyle\int_{\partial B_{R_M}} v_{r}^2\,d\mathcal{H}^{n-1}}= \sigma_1(A_{r,{R_M}}),
    \end{equation*}
    concluding the proof. Clearly the equality holds if and only if $\Omega_0=B_{R_M}$.
    \end{proof}

\subsection{Results related to $\sigma_2$ and the corresponding eigenfunctions}
 We point out that from now on we will use the following notation:  Let $\varepsilon>0$ be small enough. let $\Omega_{r_\varepsilon}= \Omega_0\setminus \overline{B}_{r_\varepsilon}$, with $B_{r_\varepsilon}$ being the ball of radius $0<{r_\varepsilon}<\varepsilon$ satisfying \eqref{eq:r=o(eps)}, such that $\overline{B}_{r_\varepsilon}\subset \Omega_0$.
\begin{proof}[Proof of Theorem \ref{thm:main}] We divide the proof in four steps.\\

\noindent\textbf{Step 1} \textit{(Upper bound for the $\limsup$)}: Firstly, we prove that \begin{equation}
        \limsup_{\varepsilon\to 0^+} \sigma_2(\Omega_{r_\varepsilon})\le \overline{\sigma}_1(\Omega_0).
    \end{equation}
Let us consider an eigenfunction corresponding to the first non-trivial Steklov eigenvalue in $\Omega_0$ and let us denote it by $\overline{u}_1$. Moreover let us denote by $u_1^{r_\varepsilon}$ a positive eigenfunction corresponding to the first Steklov-Dirichlet eigenvalue in $\Omega_{r_\varepsilon}$. Let $\eta>0$ be fixed (it will be chosen later) and let us choose $\varepsilon$ small enough, such that $A_{{r_\varepsilon},\varepsilon}\subset\{u_1^{r_\varepsilon}<\eta\}$. Let us define
    \begin{equation*}
        u^{\eta}= \max \{\eta, u_1^{r_\varepsilon}\}.
    \end{equation*}
    We stress that for $\eta$ and $\varepsilon$ small enough, since $u_1^{r_\varepsilon}|_{\partial \Omega_0}>0$, we have
    \begin{equation*}
       \;\;\;\;\; u^\eta = u_1^{r_\varepsilon} \qquad \text{on}\;\; \partial \Omega_0.
    \end{equation*}
    We now define the following test function for $\sigma_2(\Omega_{r_\varepsilon})$ \begin{equation}\label{eq:testfunc}
        \varphi = \omega^\varepsilon\frac{\overline{u}_1}{u^{\eta}},
    \end{equation}
    where $\omega_\varepsilon$ is defined in \eqref{eq:omegaeps}. This is a good test function since for $\eta,\varepsilon$ small enough, we get
    \begin{equation*}
        \int_{\partial \Omega_0}\varphi u_1^{r_\varepsilon} = \int_{\partial \Omega_0}\overline{u}_1 = 0.
    \end{equation*}
    Using \eqref{eq:testfunc} in the variational characterization \eqref{eq:carvarSD} of $\sigma_2(\Omega_{r_\varepsilon})$, we get
    \begin{equation*}
     \sigma_2(\Omega_{r_\varepsilon}) \le \frac{\ds\int_\Omega \abs{\nabla \overline{u}_1 \frac{\omega^\varepsilon}{u^{\eta}}+ \overline{u}_1\nabla \bigg(\frac{\omega^\varepsilon}{u^{\eta}}\bigg)}^2\,dx}{\ds\int_{\partial \Omega_0} \frac{\overline{u}_1^2}{(u_1^{r_\varepsilon})^2}}.
    \end{equation*}
Let us begin with the denominator. Since for fixed $\eta>0$ we have that $u_1^{r_\varepsilon}\ge\eta$ for $x \in \partial \Omega_0$, and $\overline{u}_1\in L^2(\partial \Omega_0)$, by the dominated convergence theorem we get
\begin{equation*}
    \int_{\partial \Omega_0} \frac{\overline{u}_1^2}{(u_1^{r_\varepsilon})^2} \to \frac{1}{c_{\Omega_0}^2}\int_{\partial \Omega_0} \overline{u}_1^2,
\end{equation*}
where $c_{\Omega_0}>0$ is the constant which $u_1^{r_\varepsilon}$ converges to a.e. in $\Omega_0$, appearing in Theorem \ref{prop:firsteig}.\\
Let us now estimate the numerator: 
\begin{equation*}
\begin{split}
    \int_\Omega& \abs{\nabla \overline{u}_1 \frac{\omega^\varepsilon}{u^{\eta}}+ \overline{u}_1\nabla \bigg(\frac{\omega^\varepsilon}{u^{\eta,}}\bigg)}^2\,dx\le \int_{\Omega_0} \abs{\nabla \overline{u}_1}^2 \bigg(\frac{\omega^\varepsilon}{u^{\eta}}\bigg)^2\,dx\\
    &+2\int_{\Omega_0}\frac{\overline{u}_1\omega^\varepsilon}{u^{\eta}} \nabla \overline{u}_1\cdot \nabla \bigg(\frac{\omega^\varepsilon}{u^{\eta}}\bigg)\,dx + \int_{\Omega_0}\overline{u}_1^2\abs{\nabla \bigg(\frac{\omega^\varepsilon}{u^{\eta}}\bigg)}^2\,dx\\
    &=: I_1+I_2+I_3.
    \end{split}
\end{equation*}
Since $\omega^\varepsilon\le 1$, $\overline{u}_1$ is bounded in $\Omega_0$ and $u^{\eta}\ge\eta>0$, then using the Cauchy-Schwarz inequality on $I_2$, we have that
\begin{equation*}
    I_2+I_3\le 2\frac{\max_{\Omega_0}\abs{\overline{u}_1}}{\eta}\|\nabla \overline{u}_1\|_{L^2(\Omega_0)}\|\nabla (\omega^\varepsilon/u^{\eta})\|_{L^2(\Omega_0)}+\max_{\Omega_0}\overline{u}_1^2\cdot\|\nabla (\omega^\varepsilon/u^{\eta})\|^2_{L^2(\Omega_0)}.
\end{equation*}
We stress that
\begin{equation}\label{eq:i1i2}
    I_2+I_3\le C_1\|\nabla (\omega^\varepsilon/u^{\eta})\|_{L^2(\Omega_0)}+C_2\|\nabla (\omega^\varepsilon/u^{\eta})\|_{L^2(\Omega_0)}^2,
\end{equation}
where the two constants $C_1$ and $C_2$ are positive and not depending on $\varepsilon$ (but only on $\Omega_0$ and $\eta$).
Now, using the fact that $(x-y)^2\le 2x^2+2y^2$, for every $x,y\in \mathbb R$
\begin{equation*}
\begin{split}
    \|\nabla (\omega^\varepsilon/u^{\eta})\|_{L^2(\Omega_0)} &= \int_{\Omega_0}\abs{\frac{ u^{\eta}\nabla\omega^\varepsilon-\omega^\varepsilon\nabla u^{\eta}}{(u^{\eta})^2}}^2\,dx\\
    &\le 2\int_{\Omega_0}\frac{\abs{\nabla\omega^\varepsilon}^2}{(u^{\eta})^2}\,dx+2\int_{\Omega_0}\frac{(\omega^\varepsilon)^2\abs{\nabla u^{\eta}}^2}{(u^{\eta})^4}\,dx\\
    &=\frac{2}{\eta^2}\int_{A_{{r_\varepsilon},\varepsilon}}\abs{\nabla\omega^\varepsilon}^2\,dx+\frac{2}{\eta^4}\int_{\{u_1^{r_\varepsilon}\ge \,\eta\}}\abs{\nabla u^{r_\varepsilon}_1}^2\,dx.
    \end{split}
\end{equation*}
where in the last line we have used the fact that $\omega^\varepsilon\in[0,1]$ and it is constant outside of $A_{{r_\varepsilon},\varepsilon}$, and $u^{\eta}\ge\eta$ and it is constant in $\{u_1^{r_\varepsilon}<\eta\}$. In particular 
\begin{equation*}
    \|\nabla (\omega^\varepsilon/u^{\eta})\|_{L^2(\Omega_0)} \le\frac{2}{\eta^2}\int_{\Omega_0}\abs{\nabla\omega^\varepsilon}^2 \,dx+\frac{2}{\eta^4}\int_{\Omega_0}\abs{\nabla u_1^{r_\varepsilon}}^2.
\end{equation*}
Since $\nabla u_1^{r_\varepsilon}\to 0$ in $L^2(\Omega_0)$ by Proposition \ref{prop:firsteig} and $\nabla \omega^\varepsilon\to 0$ in $L^2(\Omega_0)$ by Lemma \ref{lemma:omega}, then 
\begin{equation*}
    \|\nabla (\omega^\varepsilon/u^{\eta})\|_{L^2(\Omega_0)} \to 0 \qquad \text{as}\;\; \varepsilon\to 0^+.
\end{equation*}
This fact and \eqref{eq:i1i2} imply
\begin{equation*}
    \lim_{\varepsilon\to 0^+} I_2+I_3 \le 0. 
\end{equation*}
Let us now observe the behaviour of $I_1$ with respect to the limit. We observe that 
\begin{equation*}
    \lim_{\varepsilon\to 0^+} u^{\eta} = \max\{\eta, c_{\Omega_0}\}\ge c_{\Omega_0}, \qquad \text{a.e. in }\; \Omega_0.
\end{equation*}
Moreover, for $\eta>0$, we have
\begin{equation*}
    \abs{\nabla \overline{u}_1}^2 \bigg(\frac{\omega^\varepsilon}{u^{\eta}}\bigg)^2 \le \frac{\abs{\nabla \overline{u}_1}^2}{\eta^2}\in L^1(\Omega_0),
\end{equation*}
so that by the dominated convergence theorem we get
\begin{equation*}
    \lim_{\varepsilon\to 0^+} I_1 \le \frac{1}{c_{\Omega_0}^2}\int_{\Omega_0}\abs{\nabla \overline{u}_1}^2\,dx.
\end{equation*}
Eventually, we get
\begin{equation*}
    \limsup_{\varepsilon\to 0^+}\sigma_2(\Omega_{r_\varepsilon}) \le \lim_{\varepsilon\to 0^+} (I_1+I_2+I_3) \le \frac{\int_{\Omega_0}\abs{\nabla \overline{u}_1}^2\,dx}{\int_{\partial \Omega_0} \overline{u}_1^2} = \overline{\sigma}_1(\Omega_0).
\end{equation*}
\textbf{Step 2} \textit{(Uniform equiboundedness)}:     We stress that Step $1$ allows us to give an upper bound for $\sigma_2(\Omega_{r_\varepsilon})$ when $r_\varepsilon$ is small enough. Indeed by the definition of limit we have that for every fixed $\varepsilon_0>0$, there exists a $r_0>0$, such that for every $0<r_\varepsilon<r_0$, we get
    \begin{equation*}
        \sigma_2(\Omega_{r_\varepsilon}) \le \overline{\sigma}_1(\Omega_0)+\varepsilon_0\le \overline{\sigma}_1(B_R)+\varepsilon_0= \frac{1}{R}+\varepsilon_0,
    \end{equation*}
where we have applied the Weinstock inequality and $B_R$ is the ball having the same perimeter as $\Omega_0$.\\
    
   \noindent \textbf{Step 3} \textit{(Convergence of the eigenvalues)}:
 By \textit{Step $1$} and \textit{Step $2$}, we have that for small radii of the inner ball $B_{r_\varepsilon}$, $\sigma_2(\Omega_{r_\varepsilon})$ is uniformly bounded from above. Moreover it is clear that $\sigma_2(\Omega_{r_\varepsilon})>\sigma_1(\Omega_{r_\varepsilon})>0.$ Since the sequence of positive real number $\{\sigma_2(\Omega_{r_\varepsilon})\}_\varepsilon$ is bounded for small ${\varepsilon}>0$, there exists a subsequence, that we will denote in the same way, which converges to its $\liminf$, i.e.
\begin{equation*}
    \liminf_{\varepsilon\to 0^+}\sigma_2(\Omega_{r_\varepsilon}) = \Tilde{\sigma}.
\end{equation*}
Moreover, by the variational characterization of $\sigma_2(\Omega_{r_\varepsilon})$, if we assume $\|u_2^{r_\varepsilon}\|_{L^2(\partial \Omega_0)}=1$, we have that $\|\nabla u_2^{r_\varepsilon}\|_{L^2(\Omega_0)}\le C $ for every $\varepsilon$ small enough. Extending to $0$ in $B_{r_\varepsilon}$,
Friederich's inequality \eqref{friedin} guarantees that $\{u_2^{r_\varepsilon}\}_\varepsilon$ is equibounded in $L^2(\Omega_0)$. Therefore, there exists a subsequence, still denoted by $\{u_2^{r_\varepsilon}\}_\varepsilon$, that converges strongly in $L^2(\Omega_0)$ to some $\bar{u}\in H^1(\Omega_0)$ and such that $\nabla u_2^{r_\varepsilon} \rightharpoonup \nabla \bar{u}$ weakly in $L^2(\Omega_0)$. Moreover, by the compactness of the trace operator we have that $u_2^{r_\varepsilon}$ converges strongly in $L^2(\partial \Omega_0)$ and almost everywhere in $\partial \Omega_0$.\\
    Let us use as a test function in the weak formulation \eqref{eq:WeakSteklovdirichlet} of $\sigma_2(\Omega_{r_\varepsilon})$ the following
    \begin{equation*}
        \varphi= \omega^\varepsilon \psi, \qquad \psi \in H^1(\Omega_0),
    \end{equation*}
    getting
    
    \begin{equation}\label{eq:varcarsigma2} \int_{\Omega_0}\psi\nabla u_2^{r_\varepsilon}\cdot \nabla \omega^\varepsilon\,dx+\int_{\Omega_0}\omega^\varepsilon\nabla u_2^{r_\varepsilon}\cdot \nabla \psi\,dx = \sigma_2(\Omega_{r_\varepsilon})\int_{\partial \Omega_0} u_2^{r_\varepsilon}\psi\,d\mathcal{H}^{1}.
    \end{equation}
    Since $\nabla u_2^{r_\varepsilon}$ converges weakly in $L^2(\Omega_0)$ to $\nabla \bar{u}$ and $\psi\nabla \omega^\varepsilon$ strongly to $0$, then $\nabla \omega^\varepsilon\rightharpoonup 0$ weakly in $L^2(\Omega_0)$, implying that the first integral goes to zero. Moreover, the strong convergence in $L^2(\Omega_0)$ of $\omega^\varepsilon\nabla\psi$ to $\nabla\psi$, allows us to say that
    \begin{equation*}
        \lim_{\varepsilon\to 0^+} \int_{\Omega_0}\omega^\varepsilon\nabla u_2^{r_\varepsilon}\cdot \nabla \psi\,dx = \int_{\Omega_0}\nabla \bar{u}\cdot \nabla \psi\,dx.
    \end{equation*}
    In addition, the strong convergence of the traces of $u_2^{r_\varepsilon}$ implies that
    \begin{equation*}
        \lim_{\varepsilon\to 0^+}\int_{\partial \Omega_0} u_2^{r_\varepsilon}\psi\,d\mathcal{H}^{n-1}= \int_{\partial \Omega_0} \bar{u}\psi\,d\mathcal{H}^{n-1}.
    \end{equation*}
     Eventually, passing to the $\liminf$ in \eqref{eq:varcarsigma2}, as $\varepsilon\to 0^+$, we get
    \begin{equation}\label{eq:weaktildeu}
    \int_{\Omega_0}\nabla \bar{u}\cdot \nabla \psi\,dx = \Tilde{\sigma}\int_{\partial \Omega_0} \bar{u}\psi\,d\mathcal{H}^{n-1}, \qquad \forall\psi \in H^1(\Omega_0).
    \end{equation}
    The orthogonality condition yields the following (since the first eigenfunction $u_1^{r_\varepsilon}$ converges in $L^2(\partial \Omega_0)$ to the positive constant $c_{\Omega_0}$)
    \begin{equation}\label{eq:tracetildeu}
        0=\int_{\partial \Omega_0}u_2^{r_\varepsilon}u_1^{r_\varepsilon}\,d\mathcal{H}^{n-1}\to c_{\Omega_0}\int_{\partial \Omega_0}\bar{u}\,d\mathcal{H}^{n-1}.
    \end{equation}
     We stress that $\Tilde{\sigma}>0$ since, if it was zero, equation \eqref{eq:weaktildeu} would imply $\bar{u}$ constant, and equation \eqref{eq:tracetildeu} would tell us $\bar{u}=0$ on $\partial \Omega_0$. But this is not possible since the normalization condition holds.  These information tell us that $\Tilde{\sigma}$ belongs to the spectrum of the Neumann-to-Dirichlet operator and it is not zero, i.e. $\bar{\sigma}_1(\Omega_0)\le\bar{\sigma}$. The third step is concluded, since
    \begin{equation*}
        \overline{\sigma}_1(\Omega_0)\le\Tilde{\sigma}= \liminf_{\varepsilon\to 0^+}\sigma_2(\Omega_{r_\varepsilon}) \le \limsup_{\varepsilon\to 0^+}\sigma_2(\Omega_{r_\varepsilon}) \le \overline{\sigma}_1(\Omega_0).
    \end{equation*}
  
        \noindent \textbf{Step 4} \textit{(Convergence of the eigenfunctions)}: To summarize, we have proved that the sequence $\{u_2^{r_\varepsilon}\}$ (extended to $0$ in $B_{r_\varepsilon}$) is such that:
        \begin{enumerate}
            \item[$1)$] converges strongly in $L^2(\Omega_0)$ to $\bar{u}$, i.e. $\|u_2^{r_\varepsilon}-\bar{u}\|_{L^2(\Omega_0)}\to 0$;\\
            \item[$2)$] their gradients converge weakly in $L^2(\Omega_0)$ to $\nabla\bar{u}$;\\
            \item[$3)$] their traces converge strongly in $L^2(\partial\Omega_0)$ to $\bar{u}$, i.e. $\|u_2^{r_\varepsilon}-\bar{u}\|_{L^2(\partial\Omega_0)}\to 0$.\\
        \end{enumerate}
        In particular, by the weak formulation \eqref{eq:weaktildeu} and the fact that $\Tilde{\sigma}=\overline{\sigma}_1(\Omega_0)$, $\bar{u}$ must be an eigenfunction corresponding to $\overline{\sigma}_1(\Omega_0)$. Moreover, by the strong convergence of the traces, we have $\|\bar{u}\|_{L^2(\partial \Omega_0)}=1$. What remains, for proving the strong convergence in $H^1(\Omega_0)$, is the strong convergence of the gradients. However, this follows again from the weak formulation \eqref{eq:weaktildeu} and the strong convergence of the traces. Indeed, considering $u_2^{r_\varepsilon}$ as a test function in \eqref{eq:WeakSteklov}, we get
       \begin{equation}\label{eq:weakbaru}
    \int_{\Omega_0}\nabla \bar{u}\cdot \nabla u_2^{r_\varepsilon}\,dx = \overline{\sigma}_1(\Omega_0)\int_{\partial \Omega_0} \bar{u}u_2^{r_\varepsilon}\,d\mathcal{H}^{n-1}.
    \end{equation}
    Hence
    \begin{equation*}
        \begin{split}
            \|\nabla u_2^{r_\varepsilon}-\nabla \bar{u}\|^2_{L^2(\Omega_0)}&= \|\nabla u_2^{r_\varepsilon}\|^2_{L^2(\Omega_0)}+ \|\nabla \bar{u}\|^2_{L^2(\Omega_0)}-2\int_{\Omega_0}\nabla \bar{u}\cdot \nabla u_2^{r_\varepsilon}\,dx\\
            &= \sigma_2(\Omega_{r_\varepsilon})+ \overline{\sigma}_1(\Omega_0)-2\,\overline{\sigma}_1(\Omega_0)\int_{\partial \Omega_0} \bar{u}u_2^{r_\varepsilon}\,d\mathcal{H}^{n-1}.
        \end{split}
    \end{equation*}
    Since by \textit{Step 3} we have $\sigma_2(\Omega_{r_\varepsilon})\to \overline{\sigma}_1(\Omega_0)$ and $\int_{\partial \Omega_0} \bar{u}u_2^{r_\varepsilon}\,d\mathcal{H}^{n-1}\to \int_{\partial \Omega_0} \bar{u}^2\,d\mathcal{H}^{n-1}=1$, we arrive to the conclusion
    \begin{equation*}
        \|\nabla u_2^{r_\varepsilon}-\nabla \bar{u}\|_{L^2(\Omega_0)}\to 0,
    \end{equation*}
    and therefore the convergence in $H^1(\Omega_0)$.
\end{proof}

\begin{proof}[Proof of Corollaries \ref{thm:isoperimetric1}-\ref{thm:isoperimetric2}:] 
We prove it only in the case of a perimeter constraint, since the one with a measure constraint follows identically applying the Brock inequality \eqref{eq:brock} instead of the Weinstock inequality \eqref{eq:weinstock} .\\
Let $R_P>0$ be the radius of the ball having the same perimeter as $\Omega_0$. If $\Omega_r = A_{r,R_P}$ there is nothing to prove. Let us suppose that $\Omega_r \neq A_{r,R_P}$. By Theorem \eqref{thm:main}, we have that 
\begin{equation}\label{eq:deflim}
    \forall \varepsilon>0, \exists r_2^P : \; \forall 0<r<r_2^P \implies \sigma_2(\Omega_r)<\overline{\sigma}_1(\Omega_0)+\varepsilon.
\end{equation}
Moreover, since the Weinstock inequality \eqref{eq:weinstock} holds as an equality if and only if $\Omega_0=B_{R_P}$, there exists a  $\overline{\varepsilon}>0$, such that
\begin{equation}\label{eq:weinspace}
    \overline{\sigma}_1(\Omega_0)\le \overline{\sigma}_1(B_{R_P})-\overline{\varepsilon}.
\end{equation}
    Now we can choose $\varepsilon\le\overline{\varepsilon}$. Putting together \eqref{eq:deflim} and \eqref{eq:weinspace}, we conclude that
    \begin{equation*}
        \sigma_2(\Omega_r)<\overline{\sigma}_1(B_{R_P})+(\varepsilon-\overline{\varepsilon})< \sigma_2(A_{r,R_P}),
    \end{equation*}
where in the last inequality we have applied Proposition \ref{prop:convergenceeigrad}. 
    
\end{proof}

\section*{Acknowledgments}
This work has been partially supported  by GNAMPA of INdAM.

\section*{Conflicts of interest and data availability statement}
The author declares that there is no conflict of interest. Data sharing not applicable to this article as no datasets were generated or analyzed during the current study.

\bibliographystyle{plain}
\bibliography{biblio}

\end{document}